\documentclass[11pt]{amsart}

\usepackage{graphicx}
\usepackage{subcaption}
\usepackage{booktabs}

\usepackage{rainstorm}

\usepackage{hyperref}
\usepackage[noabbrev,capitalise]{cleveref}

\makeatletter
\@ifpackageloaded{newtxmath}{}{\usepackage{amssymb}}

\theoremstyle{plain}
\newtheorem{theorem}{Theorem}[section]
\AddToHook{env/theorem/begin}{\crefalias{section}{theorem}}

\newtheorem{proposition}[theorem]{Proposition}
\newtheorem{corollary}[theorem]{Corollary}
\newtheorem{lemma}[theorem]{Lemma}
\newtheorem{conjecture}[theorem]{Conjecture}
\AddToHook{env/proposition/begin}{\crefalias{theorem}{proposition}}
\AddToHook{env/corollary/begin}{\crefalias{theorem}{corollary}}
\AddToHook{env/lemma/begin}{\crefalias{theorem}{lemma}}
\AddToHook{env/conjecture/begin}{\crefalias{theorem}{conjecture}}

\newtheorem*{theorem*}{Theorem}
\newtheorem*{conjecture*}{Conjecture}

\newtheorem*{corollary*}{}

\newtheorem{question}{Question}

\crefname{question}{Question}{Questions}
\Crefname{question}{Question}{Questions}

\theoremstyle{definition}
\newtheorem{definition}[theorem]{Definition}
\newtheorem{example}[theorem]{Example}

\AddToHook{env/definition/begin}{%
  \crefalias{theorem}{definition}%
  \pushQED{\qed}%
}
\AddToHook{env/definition/end}{\popQED}
\AddToHook{env/example/begin}{\crefalias{theorem}{example}}
\AddToHook{env/notation/begin}{\crefalias{theorem}{notation}}

\theoremstyle{remark}

\crefname{theorem}{Theorem}{Theorems}
\Crefname{theorem}{Theorem}{Theorems}
\crefname{proposition}{Proposition}{Propositions}
\Crefname{proposition}{Proposition}{Propositions}
\crefname{corollary}{Corollary}{Corollaries}
\Crefname{corollary}{Corollary}{Corollaries}
\crefname{lemma}{Lemma}{Lemmas}
\Crefname{lemma}{Lemma}{Lemmas}
\crefname{conjecture}{Conjecture}{Conjectures}
\Crefname{conjecture}{Conjecture}{Conjectures}
\crefname{definition}{Definition}{Definitions}
\Crefname{definition}{Definition}{Definitions}
\crefname{example}{Example}{Examples}
\Crefname{example}{Example}{Examples}
\crefname{notation}{Notation}{Notations}
\Crefname{notation}{Notation}{Notations}

\newcommand{\tri}{\mathcal{T}}
\newcommand{\lst}{\operatorname{LST}}
\newcommand{\manifold}{\mathcal{M}}
\newcommand{\sone}{\mathbb{S}^1}
\newcommand{\stwo}{\mathbb{S}^2}
\newcommand{\sthree}{\mathbb{S}^3}
\newcommand{\RR}{\mathbb{R}}
\newcommand{\ZZ}{\mathbb{Z}}
\newcommand{\QQ}{\mathbb{Q}}

\newcommand{\lensp}{\mathfrak{p}}
\newcommand{\lensq}{\mathfrak{q}}
\newcommand{\pp}{\mathcal{P}}
\newcommand{\qq}{\mathcal{Q}}

\begin{document}

\title{Torus knots as loop-edges in three-sphere triangulations}

\author{Lezhi Lin}
\email{lezhi.lin@sydney.edu.au}
\address{School of Mathematics and Statistics F07, The University of Sydney, NSW 2006 Australia}
\author{Jonathan Spreer}
\email{jonathan.spreer@sydney.edu.au}
\address{School of Mathematics and Statistics F07, The University of Sydney, NSW 2006 Australia}

\subjclass[2020]{Primary 57K10; Secondary 57Q15, 57K30}
\keywords{torus knots, triangulations of 3-manifolds, layered triangulations, triangulation complexity}

\begin{abstract}
We present a procedure that, given a pair of coprime integers, produces a one-vertex triangulation of the three-sphere containing the corresponding torus knot as an edge. The construction is canonical, and produces a three-sphere triangulation from a pair of minimal layered triangulations of the solid torus.

The size of the triangulation is, in the worst case, linear, but can be as small as logarithmic in the crossing number of the knot, motivating this triangulation size as an organising principle for knots, alternative to the crossing number.
\end{abstract}

\maketitle

\section{Introduction}
\label{sec:introduction}

A \emph{torus knot} $T(\pp,\qq)$ in the three-sphere is a knot drawn on the surface of a standardly embedded torus, winding $\pp$ times in meridional direction and $\qq$ times in the longitudinal direction (see \Cref{fig:slope_3d}). Among non-trivial knots they are about as well understood as one could hope, being classified up to ambient isotopy by the coprime pair $(\pp,\qq)$ and accompanied by explicit formulas for many classical invariants~\cite{Kauffman1987,Murasugi1991,Rolfsen2003,Teragaito2004}.

Torus knots also play a significant role in Thurston's geometric classification of prime knot complements~\cite[Corollary~2.5]{Thurston1982}, which partitions knots into torus, satellite and hyperbolic types. As a non-hyperbolic, explicitly parametrised family, torus knots serve as the tractable base case against which other knots are understood.
Beyond classification, torus knots occur 
in the construction of closed manifolds from surgery diagrams~\cite{Lickorish1962,Wallace1961}, in the Heegaard structure of Seifert fibered spaces~\cite{Moriah1988,Schultens1993}, in surgery realisations of lens spaces~\cite{Moser1971}, and in the classification of the knots admitting such surgeries~\cite{KronheimerMrowkaOzsvathSzabo2007,OzsvathSzabo2005}.

We approach torus knots from the combinatorial side of the theory: How efficiently can $T(\pp,\qq)$ be presented as a loop edge in a one-vertex \emph{triangulation} of $\sthree$? 

By Moise's theorem, every compact $3$-manifold can be triangulated~\cite{Moise1952}. The smallest such description, the \emph{complexity} of the manifold, is a natural combinatorial measure of its topological complexity. It is the organising invariant for systematic enumeration of $3$-manifolds~\cite{Burton2011,Matveev2007}.
This concept of complexity transfers naturally to knotted edge-cycles inside triangulations of $3$-manifolds. Rather than measuring a knot $K \subset \manifold$ by its crossing number, or by the complexity of its exterior, one can ask how few tetrahedra are needed to triangulate $\manifold$ so that $K$ appears as a subcomplex, ideally as a single loop-edge of a one-vertex triangulation. This combinatorial counterpart of the crossing number the \emph{triangulation complexity} of the pair $(\manifold, K)$ was put forward in recent work, see~\cite{IbarraMathewsPurcellSpreer2024}, and is interesting for two reasons. First, it couples two quantities -- the size of an ambient triangulation and the geometry of an embedded curve -- under a single combinatorial roof, thus opening a bridge between knot theory and the algorithmic theory of $3$-manifolds. Second, it is the natural setting for manipulating a knot inside a triangulated $3$-manifold without carrying separate combinatorial data, used in recent work on knot factorisation, recognition, and certification~\cite{HeSedgwickSpreer2025,Lackenby2020}.

For pairs $(\manifold, K)$, exact values are known only in a handful of small examples, and matching lower bounds remain elusive. This is unsurprising: even for closed $3$-manifolds, explicit upper and lower complexity bounds have been established only for distinguished families~\cite{LackenbyPurcell2024,JacoRubinsteinSpreerTillmann2020, JacoRubinsteinSpreerTillmann2025,JacoRubinsteinTillmann2009}.

For torus knots in $\sthree$, the natural ambient triangulations are already at hand. Jaco and Rubinstein's theory of \emph{layered triangulations} of solid tori~\cite{JacoRubinstein2006}, reviewed in \Cref{sec:layered-triangulations}, assigns to each coprime pair of integers an explicit one-vertex triangulation of a solid torus whose meridian has a prescribed slope and whose size is governed by the Euclidean algorithm. The framework and its extensions have since been built upon to prove that some layered triangulations of lens spaces are minimal~\cite{JacoRubinsteinSpreerTillmann2025,JacoRubinsteinTillmann2009}, to construct one-vertex triangulations of arbitrary Heegaard splittings~\cite{HeMorganThompson2025}, and to exploit parameterised algorithms for $3$-manifold invariants organised by triangulation treewidth, as outlined in~\cite{HuszarSpreer2019}.

Gluing two such layered solid tori along their common one-vertex boundary torus realises, in a single combinatorial framework, every closed $3$-manifold admitting a genus-$1$ Heegaard splitting, that is, every lens space, $\stwo \times \sone$, and the three-sphere itself~\cite[Section 1.5]{Saveliev2011}. The picture is at once topological (a genus-$1$ Heegaard splitting), combinatorial (a path through Farey triangles), and arithmetic (a run of the Euclidean algorithm), and the three points of view amplify rather than substitute for each other. In this framework, an embedding $T(\pp,\qq) \hookrightarrow \sthree$ amounts to a choice of two layered solid tori whose meridian slopes are in compatible directions. 

\subsection*{Our Contribution}

Exploiting this correspondence, we exhibit in~\Cref{sec:torus-knots}, for every torus knot, a canonical one-vertex triangulation of $\sthree$ realising it as a single edge, of size at most linear in $\max\{\pp,\qq\}$ and hence at most linear in the crossing number of the knot. We conjecture our examples to realise the triangulation complexities of all torus knots $T(\pp,\qq)$. 

Moreover, we observe that, for some infinite families of torus knots, our triangulations are in size a logarithmic function in the crossing numbers of the underlying knots, see~\Cref{tab:canonical-fibonacci} for one such family. This is significant for the following reason: many algorithms in knot theory (see~\cite{Lackenby2020} for an overview) take as input a knot diagram, say, with $n$ crossings, turn this diagram into a triangulation of the knot complement with $O(n)$ tetrahedra, and then run an algorithm on it that is exponential in $n$~\cite{HassLagariasPippenger1999}. Families such as the one from~\Cref{tab:canonical-fibonacci} provide an example where an exponential speed-up can be introduced into this process. 

Code to produce our triangulations is available from \cite{Lin2026Github}.

\subsection*{Acknowledgement}

The authors would like to thank Stephan Tillmann for helpful advice on how to illustrate this article, and Alexander He and Marc Lackenby for sketching the satellite construction. The second author was partially supported by the Australian Research Council under the Discovery Project theme, grant number DP220102588. The authors would also like to acknowledge the University of Sydney unit MATH1964 for indirectly starting this collaboration.

The authors acknowledge the use of GPT-5.6 Sol to check the manuscript for typos, improve its writing style, and assist in drafting and refining some of the source code. 

\section{Preliminaries}
\label{sec:preliminaries}

\subsection{Knots and 3-manifolds}
\label{sec:pre-knots-manifolds}

We recall a number of standard notions from knot theory and $3$-manifold topology. For a comprehensive treatment the reader may consult Rolfsen~\cite{Rolfsen2003} and Hempel~\cite{Hempel1976}.

A \emph{knot} is a smooth embedding $K \colon \sone \hookrightarrow \sthree$, considered up to ambient isotopy. The \emph{crossing number} of $K$, denoted $c(K)$, is the minimum number of crossings over all planar diagrams of $K$. A fundamental object associated with a knot is the closure of the complement of a small open tubular neighbourhood:

\begin{definition}
\label{def:knot-exterior}
Let $K \subset \sthree$ be a knot and let $\nu(K)$ be an open tubular neighbourhood of $K$. The \emph{exterior} of $K$ is the compact $3$-manifold
\[
E(K) = \sthree \setminus \nu(K).
\]
Its boundary $\partial E(K)$, a torus, carries a natural \emph{meridian-longitude basis} $\{\mu, \lambda\}$: the \emph{meridian} $\mu$ is the unique slope (up to isotopy on $\partial E(K)$) bounding a disk in $\nu(K)$, and the \emph{homological longitude} $\lambda$ is the slope that meets $\mu$ in a single point and is null-homologous in $E(K)$.
\end{definition}

With a basis for $H_1(\partial E(K); \ZZ)$ in hand, every essential simple closed curve on the boundary torus can be encoded by a pair of coprime integers. This leads to the following notion.

\begin{definition}
\label{def:slope}
Let $T$ be a torus equipped with a basis $\{[\mu], [\lambda]\}$ of $H_1(T; \ZZ)$. A \emph{slope} on $T$ is an isotopy class of essential simple closed curves on $T$.
Each slope has a representative given by a unique primitive class $b\,[\mu] + a\,[\lambda]$ with $(a, b) \in \ZZ^2$, $\gcd(a, b) = 1$. This is well-defined up to simultaneous sign change. Slopes naturally biject into $\QQ \cup \{\infty\}$ via the assignment $(a, b) \mapsto a/b$. 
\end{definition}

\begin{figure}[t!]
    \centering
    \begin{subfigure}[b]{0.44\textwidth}
        \centering
        \includegraphics[
            width=0.85\textwidth,
            alt={Curve winding twice meridionally and three times longitudinally on a torus; blue arrows mark the basis identifications. Details in the caption.}
        ]{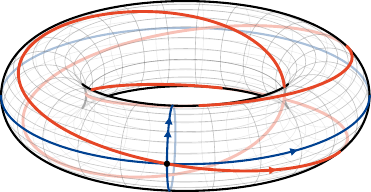}
        \vspace{6pt}
        \caption{$T(2,3)$ on the torus.}
        \label{fig:slope_3d}
    \end{subfigure}\hfill
    \begin{subfigure}[b]{0.54\textwidth}
        \centering
        \includegraphics[
            scale=1,
            alt={Universal-cover grid with the slope vector from the origin to three comma two, representing two mu plus three lambda. Details in the caption.}
        ]{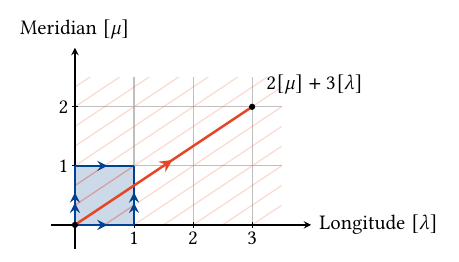}
        \caption{Lift to the universal cover.}
        \label{fig:slope_2d}
    \end{subfigure}

    \caption{The slope $3/2$, representing $2[\mu]+3[\lambda]$. (\ref{sub@fig:slope_3d}) On the torus, the ochre curve winds twice meridionally and three times longitudinally; the blue longitude and meridian loops carry one and two identification arrows, respectively. (\ref{sub@fig:slope_2d}) In the universal cover, the same class is the vector $(3,2)$.}
    \label{fig:slope}
\end{figure}

Slopes provide the gluing recipe for constructing closed $3$-manifolds out of manifolds with torus boundary. The central operation here is given by the notion of {\em Dehn filling}.

\begin{definition}
\label{def:dehn-filling}
Let $\manifold$ be a compact orientable $3$-manifold whose boundary contains a torus component $T$, and let $\alpha$ be a slope on $T$. The \emph{Dehn filling} of $\manifold$ along $\alpha$, written $\manifold(\alpha)$, is the $3$-manifold obtained by gluing a solid torus to $T$ so that its meridian is identified with $\alpha$.
\end{definition}

The following classical result, due independently to Lickorish~\cite{Lickorish1962} and Wallace~\cite{Wallace1961}, asserts that Dehn surgery suffices to produce every closed $3$-manifold.

\begin{theorem}[Lickorish--Wallace]
\label{thm:lickorish-wallace}
Every closed, connected, orientable $3$-manifold can be obtained by Dehn surgery on a link in $\sthree$.
\end{theorem}

The simplest closed $3$-manifolds arising from Dehn filling are the lens spaces, each of which admits a genus-$1$ Heegaard splitting into two solid tori. They will play a central role in the constructions of~\Cref{sec:layered-triangulations}.

\begin{definition}[{\cite{Tietze1908}}]
\label{def:lens-space}
Let $\lensp , \lensq \ge  0$, and $\gcd(\lensp, \lensq) = 1$. The \emph{lens space} $L(\lensp, \lensq)$ is the closed orientable $3$-manifold obtained by gluing two solid tori along their boundary tori via a homeomorphism that sends the meridian of one solid torus to the slope $\lensp/\lensq$ on the boundary of the other, with respect to the meridian-longitude basis.
\end{definition}

Lens spaces are classified up to homeomorphism by the following result due to \cite{Reidemeister1935,Whitehead1941}, see also{~\cite[Section~1.5]{Saveliev2011}}.

\begin{theorem}
\label{thm:lens-class}
Let $\lensp, \lensp', \lensq, \lensq' \ge 0$ with $\gcd(\lensp, \lensq) = \gcd(\lensp', \lensq') = 1$.
Then $L(\lensp, \lensq)$ and $L(\lensp', \lensq')$ are homeomorphic if and only if $ \lensp = \lensp'$ and $\lensq' \equiv \pm \lensq^{\pm 1} \pmod {\lensp}$. Moreover, an orientation-preserving homeomorphism exists if and only if $\lensq' \equiv \lensq^{\pm 1} \pmod {\lensp}$.
\end{theorem}

\subsection{Triangulations}
\label{sec:pre-triangulations}

Throughout this article, we work with \emph{generalised triangulations} in the sense of Matveev~\cite{Matveev2007} and Thurston~\cite[Section~3.2]{Thurston1997}.

\begin{definition}
\label{def:triangulation}
A \emph{triangulation} $\tri$ of a closed $3$-manifold $\manifold$ is a realisation of $\manifold$ as a quotient of finitely many tetrahedra by affine identifications of their $2$-faces in pairs. Such identifications are permitted to pair distinct faces, edges, or vertices of the same tetrahedron, so long as no edge is identified with itself in reverse, and the resulting quotient is homeomorphic to $\manifold$. We write $\tri^{(k)}$ for the set of $k$-cells ($k = 0, 1, 2, 3$) of~$\tri$.
\end{definition}

The local combinatorial structure at each vertex is captured by the vertex link: the \emph{link} of a vertex $v \in \tri^{(0)}$ is the boundary of a small regular neighbourhood of $v$. For a triangulation of a closed $3$-manifold, every vertex link is necessarily a $2$-sphere. A triangulation is called a one-vertex triangulation if it has exactly $1$ vertex, that is, if $|\tri^{(0)}| = 1$. We have the following result.

\begin{theorem}[Moise~\cite{Moise1952}]
\label{thm:moise}
Every compact $3$-manifold admits a triangulation.
Moreover, any two triangulations of the same compact $3$-manifold admit a common subdivision.
\end{theorem}

Since every compact $3$-manifold can be triangulated, it is natural to ask how efficiently. The \emph{complexity} of a closed $3$-manifold $\manifold$ is the minimum number of tetrahedra over all triangulations of $\manifold$; see Matveev~\cite[Section 2.1]{Matveev2007} for a systematic treatment, where his notion of complexity differs from ours for $\sthree$, $\RR P^3$, $\stwo \times \sone$ and $L(3,1)$.

\section{Layered triangulations of the three-sphere}
\label{sec:layered-triangulations}

Layered triangulations of the solid torus were first described by Jaco and Rubinstein in \cite{JacoRubinstein2006}. We briefly review their construction.

\subsection{Layered triangulations of solid tori and lens spaces}

We begin with the standard base case for the layered construction: the solid torus obtained from the $0$-tetrahedron M\"obius band, as the core of the one-tetrahedron layered solid torus, see \Cref{fig:mobius-base} for an illustration. This can be realised by taking a single tetrahedron $\Delta=(0,1,2,3)$, and identifying triangles $(0,1,2)$ with $(1,2,3)$ by sending $0$ to $1$, $1$ to $2$, and $2$ to $3$. The meridian disk intersects the three edges of the boundary of this one-tetrahedron solid torus $1$, $2$, and $3$ times, respectively. More general layered solid tori are then built by successively attaching tetrahedra along two boundary triangles, covering a boundary edge of the previous complex.

\begin{figure}[t!]
    \centering
    \includegraphics[
        scale=1,
        alt={Diagram of the Moebius band thickened to a solid torus and cut open to a triangle and to a quadrilateral. Details in the caption.}
    ]{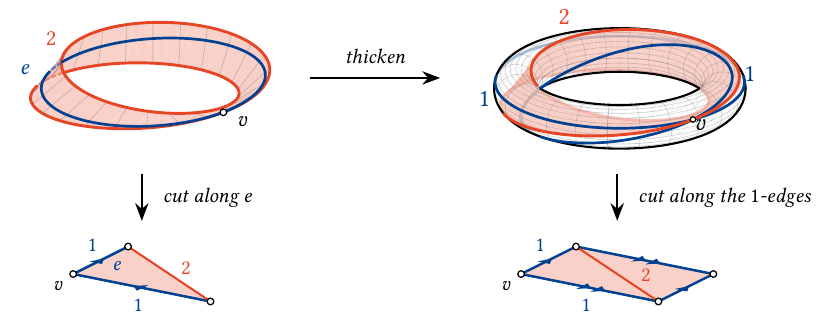}
    \caption{The $0$-tetrahedron M\"obius band (upper left), its thickening to a solid torus (upper right), and cuts along $e$ and the $1$-edges yielding the M\"obius triangle and one-vertex boundary quadrilateral below. Blue arrows pair the $1$-edges, ochre marks the $2$-edge, and $v$ is the unique vertex.}
    \label{fig:mobius-base}
\end{figure}

\begin{definition}[Layering step]
\label{def:layering-step}
Let $\tri$ be a triangulation of the solid torus with boundary $\partial \tri$ its unique one-vertex $2$-triangle triangulation, shown in Figure~\ref{fig:mobius-base}. A \emph{layering step} over boundary edge $e$ consists of attaching a tetrahedron to $\tri$ along the two boundary triangles of $\partial \tri$ covering boundary edge $e$.     
\end{definition}

By construction, the triangulation $\tri'$ resulting from a layering step on a one-vertex boundary solid torus is again a triangulated solid torus. On the boundary, this operation replaces the edge $e$ by the opposite diagonal in the quadrilateral formed by the two incident triangles with diagonal $e$. Combinatorially, the boundary is still the same one-vertex triangulation of the torus, but its embedding has changed.
See \Cref{fig:layering-step} for an illustration.

\begin{figure}[t!]
    \centering
    \begin{subfigure}[b]{0.24\textwidth}
        \centering
        \includegraphics[
            scale=1,
            alt={Panel one. A tetrahedron lies above two boundary faces sharing edge e; dashed arrows show the intended gluing. Details in caption.}
        ]{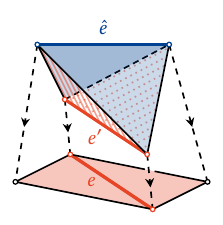}
        \caption{Old boundary edge $e$.}
        \label{fig:layering1}
    \end{subfigure}\hfill
    \begin{subfigure}[b]{0.24\textwidth}
        \centering
        \includegraphics[
            scale=1,
            alt={Panel two. A tetrahedron is positioned so that its edge e prime is matched with the old edge e. Details in the caption.}
        ]{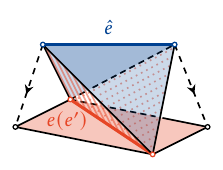}
        \caption{Matching $e'$ to $e$.}
        \label{fig:layering2}
    \end{subfigure}\hfill
    \begin{subfigure}[b]{0.24\textwidth}
        \centering
        \includegraphics[
            scale=1,
            alt={Panel three. The tetrahedron is glued along the two boundary faces incident to the edge e. Details in the caption.}
        ]{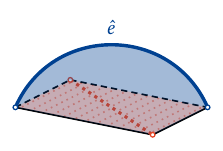}
        \caption{Gluing incident faces.}
        \label{fig:layering3}
    \end{subfigure}\hfill
    \begin{subfigure}[b]{0.24\textwidth}
        \centering
        \includegraphics[
            scale=1,
            alt={Panel four. After gluing the old edge e lies inside the solid torus and the new boundary edge e hat appears. Details in the caption.}
        ]{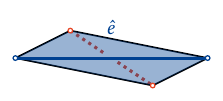}
        \caption{New boundary edge $\hat e$.}
        \label{fig:layering4}
    \end{subfigure}

    \caption{A layering step on a triangulated torus (read left to right): dashed arrows match $e'$ to $e$, the two hatched incident faces are glued, and the dotted old diagonal becomes interior while $\hat e$ is the new boundary edge. The move preserves the solid torus.}
    \label{fig:layering-step}
\end{figure}

While a layering step changes how the boundary triangulation is embedded, it preserves the underlying manifold, which remains a solid torus. Iterating this move, by choosing a boundary edge for each layering step, produces a multi-parameter family of triangulated solid tori where the embedding of the one-vertex boundary triangulation can be deduced from the choice of edges to be layered over in each step. Note that there are infinitely many layered solid tori with the meridian disk intersecting boundary edges $p$, $q$ and $p+q$ times (where $(p,q)$ is a reduced pair and $p \le q$). Of these infinitely many triangulations there is a unique one that is smallest, which we denote by $\lst(p, q)$ (it is precisely the triangulation in this family where we never layer over a newly introduced edge). In \cite{JacoRubinstein2006}, these triangulations are called \emph{minimally layered}.

To organise these embeddings of the boundary triangulation, we record the slopes of (oriented) boundary edges as integer pairs, as defined in \Cref{def:slope}, see \Cref{fig:slope}.

The Farey tessellation of the hyperbolic disk, shown in~\Cref{fig:farey-path-disk}, provides a convenient combinatorial model for these slopes: its vertices are the elements of $\QQ\cup\{\infty\}$, and the three boundary edge slopes arising from a one-vertex torus triangulation are pairwise adjacent, determining a triangle in the Farey tessellation (see~\Cref{fig:farey-triangle-boundary}). A layering step then corresponds to passing to an adjacent triangle.
The associated tree of slopes, shown in~\Cref{fig:farey-path-tree}, is the restriction to $[0,1]$ of the Stern--Brocot tree~\cite{Stern1858,Brocot1861}, and the subtraction form of the Euclidean algorithm used throughout is the classical \emph{anthyphairesis}~\cite{Fowler1987}.
In~\cite{JacoRubinstein2006}, this same mechanism is described by the closely related $L$-graph. The following theorem makes this precise.

\begin{theorem}[Farey tessellation and Farey paths, see {\cite[Chapter 1]{HatcherTopologyOfNumbers}}]
\label{thm:farey}
The vertices of the \emph{Farey tessellation} $\alpha=a/b$ and $\beta=c/d$, written in reduced form, are joined by an edge if and only if the corresponding slopes intersect minimally on the torus, equivalently if
\[
|a\,d - b\,c|=1.
\]
Triples of pairwise adjacent slopes determine Farey triangles, and a path through adjacent Farey triangles is called a \emph{Farey path}. A Farey path between two slopes $a/b$ and $c/d$ is a sequence of adjacent Farey triangles whose first triangle contains $a/b$ and whose last triangle contains $c/d$.
\end{theorem}

\begin{figure}[t!]
    \centering
    \includegraphics[
        scale=1,
        alt={Diagram matching a one-vertex torus quadrilateral with slopes zero over one, one over one, and one over two to a Farey triangle. Details in caption.}
    ]{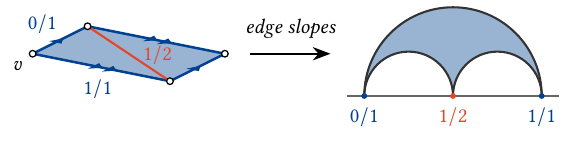}
    \caption{The three boundary edge slopes of a one-vertex torus triangulation ($0/1$, $1/1$, and diagonal $1/2$; left) span a Farey triangle (right).}
    \label{fig:farey-triangle-boundary}
\end{figure}

Every embedding of the one-vertex triangulation of the torus determines a Farey triangle via the slopes of its three edges, and a single layering step replaces one slope by the opposite vertex of an adjacent Farey triangle. Accordingly, a sequence of layering moves determines a Farey path, and conversely a boundary slope can be tracked by following such a path in the tessellation.

For a prescribed reduced slope $a/b$, we consider a Farey path from the vertex $a/b$ to $1/1$, with each edge corresponding to a single layering step. The subtraction form of the Euclidean algorithm provides precisely such a path, and hence yields the following extension result \cite[Theorem 4.3]{JacoRubinstein2006}.

\begin{theorem}
\label{thm:euclidean-alg}
    Let $\partial T$ be a one-vertex triangulation embedded on the boundary of a solid torus. Then $\partial T$ extends to a layered triangulation of the solid torus by a sequence of layering steps determined by the Euclidean algorithm.
\end{theorem}

Rather than going through a proof of \Cref{thm:euclidean-alg}, we illustrate the statement by an example.

\begin{example}
We record how many times the boundary of the meridian disk intersects the three edges of $\partial T$ as the triple $(a,b,a+b)$.
We then apply the Euclidean algorithm to $(b,a)$ in subtraction form.  
This produces a sequence of triples which, when read in reverse, gives a layering sequence from the base case $(1,1,2)$ with vertices $(0/1,1/1,1/2)$, that is, the $0$-tetrahedron M\"obius band $\lst(1,1)$ inside the one-tetrahedron layered solid torus $\lst(1,2)$, to the required boundary triangulation. See \Cref{fig:farey-path-pq} for a visualisation of this process for $a=3$ and $b=5$. 

Choosing $(0/1,1/1,1/2)$ as base triangle in \Cref{fig:farey-path-pq} is motivated by our choice to start the layering process with the core M\"obius band (see~\Cref{fig:mobius-base}). In particular, we do not consider any layered solid tori with creased $3$-cells at their core, see \cite[Section 4.2]{JacoRubinstein2006} for details about this other case.
\end{example}

\begin{figure}[t!]
    \centering
    \begin{subfigure}[b]{0.44\textwidth}
        \centering
        \includegraphics[
            width=0.95\textwidth,
            alt={Poincare disk of the Farey tessellation with two chains of shaded triangles. Four shaded Farey triangles form a path from one over one to five eighths; a reflected companion ends at minus three fifths. Details in the caption.}
        ]{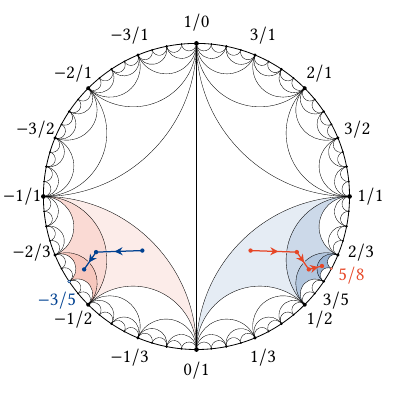}
        \caption{Farey path in the Poincar\'e disk.}
        \label{fig:farey-path-disk}
    \end{subfigure}
    \hfill
    \begin{subfigure}[b]{0.54\textwidth}
        \centering
        \includegraphics[
            width=0.95\textwidth,
            alt={Farey tree diagram of the Euclidean traversal from the base triangle to five eighths, with three-edge dual path through four shaded Farey triangles. Details in caption.}
        ]{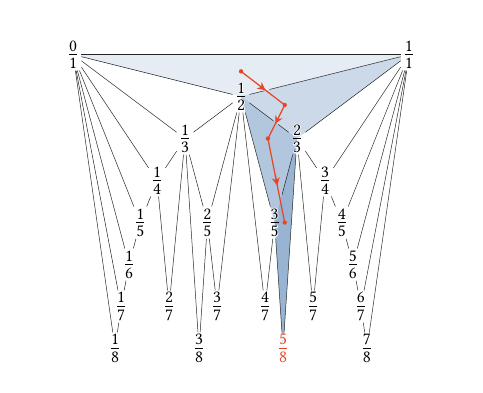}
        \caption{Euclidean traversal in the Farey tree.}
        \label{fig:farey-path-tree}
    \end{subfigure}
    
    \caption{The Farey path to $5/8$ encoding the Euclidean construction of $\lst(3,5)$. (\ref{sub@fig:farey-path-disk}) Four adjacent shaded triangles forming a path ending at $5/8$ in the Poincar\'e disk. The left path is the reflected companion ending at $-3/5$. (\ref{sub@fig:farey-path-tree}) The three-edge dual path records the same four triangle states.}
    \label{fig:farey-path-pq}
\end{figure}

Altogether,~\Cref{def:layering-step}, ~\Cref{thm:farey,thm:euclidean-alg} give three complementary descriptions of the same construction, from a topological, combinatorial, and arithmetic point of view, respectively. This understanding becomes especially useful when we pass from solid tori to closed manifolds obtained from gluing two layered solid tori together. Necessarily, such a closed manifold must admit a genus-$1$ Heegaard splitting, and hence is a lens space, $\stwo \times \sone$, or $\sthree$. The precise slope data accumulated during the layering process determines the corresponding gluing matrix and hence the exact topological type. Layered triangulations of solid tori therefore furnish a convenient entry point to triangulations of lens spaces.

\subsection{Three-sphere triangulations}

We now specialise the discussion to the case where two layered solid tori glue together to form a triangulation of the three-sphere. This happens precisely when the meridional curves of the two layered solid tori intersect on the common boundary torus in a single point. This condition is captured by a simple determinant criterion, see, for instance \cite[Section 1.5]{Saveliev2011}.

\begin{lemma}
    \label{lem:layered3sphere}
    The closed $3$-manifold obtained by gluing $\lst(p, q)$ and $\lst(r, s)$ along their boundary triangulations, such that the $p$-edge is glued to the $r$-edge, the $q$-edge to the $s$-edge, and the $(p+q)$-edge to $(r+s)$-edge, is homeomorphic to $\sthree$ if and only if $|p\,s - q\,r| = 1$.
\end{lemma}

\begin{proof}[Proof sketch]
In layered solid tori $\lst(p,q)$ and $\lst(r,s)$, the respective meridians have slopes $p/q$ and $r/s$ on their common boundary torus. The quantity $|p\,s - q\,r|$ is precisely the geometric intersection number of these slopes.  Thus the condition $|p\,s - q\,r| = 1$ says that the two meridians intersect once, and hence form a basis for the first homology of the boundary torus.

Topologically, gluing two solid tori so that their meridians intersect once gives the standard genus-$1$ Heegaard splitting of $\sthree$. Conversely, if the union of $\lst(p,q)$ and $\lst(r,s)$ is $\sthree$, then the two meridional slopes must intersect once on the common boundary torus, so the determinant condition follows. 
\end{proof}

In what follows, we refer to the triangulations of the three-sphere from \Cref{lem:layered3sphere} as {\em layered triangulations of  $\sthree$}. Note that this definition excludes triangulations of  $\sthree$ having a layered structure containing a snapped ball, since we require our layered solid tori to each contain a core M\"obius band.

\section{Torus knots in three-sphere triangulations}
\label{sec:torus-knots}

We start with the following straightforward observation.

\begin{proposition}
\label{prop:all-edges-torus-knots}
Let $\tri$ be a layered triangulation of $\sthree$, and let $e \in \tri^{(1)}$ be an edge. Then $e$ represents a torus knot in $\sthree$.
\end{proposition}

\begin{proof}
Since $\tri$ is a one-vertex triangulation of $\sthree$, every edge $e \in \tri^{(1)}$ is a simple closed curve, and hence a knot in $\sthree$. Moreover, if $\tri$ has at least two tetrahedra, $e$ is a boundary edge for some decomposition of $\tri$ into two layered solid tori. To see this, note that in the base solid torus all edges lie on its torus boundary, and each layering step introduces exactly $1$ new edge which is a boundary edge when created. At every stage the boundary torus $T \subset \sthree$ separates $\sthree$ into two solid tori (each with at least $1$ tetrahedron), forming a genus-$1$ Heegaard splitting (in particular, each of the solid tori is unknotted in $\sthree$). Altogether, $e$ is a non-trivial closed curve on an unknotted torus $T \subset \sthree$, and hence a torus knot. This leaves the $1$-tetrahedron triangulation of $\sthree$ that is a decomposition into a $1$-tetrahedron layered solid torus, and a $0$-tetrahedron M\"obius band. In this case we can verify that the two edges of this triangulation form the trivial knot and the trefoil (both are torus knots). 
\end{proof}

The following proposition identifies a torus knot of a specific type in our layered $\sthree$-triangulations.

\begin{proposition}
\label{prop:torus-knot-edge}
Layered triangulation $\tri = \lst(p,q) \, \cup \, \lst(r,s)$ of $\sthree$ contains an edge $e \in \tri^{(1)}$ representing the torus knot $T(p+q,\,r+s)$.
\end{proposition}

\begin{proof}
Let $T \subset \tri$ be the one-vertex torus at the interface of $\lst(p,q)$ and $\lst(r,s)$. By construction of $\tri$, the meridian disks of $\lst(p,q)$ and $\lst(r,s)$ intersect once in $\tri$. Also, by construction, the meridian disks of $\lst(p,q)$ and $\lst(r,s)$ intersect the three edges of the common boundary torus $p$, $q$, and $p+q$ times, and $r$, $s$, and $r+s$ times, respectively. It follows that the three edges of $T$ must form torus knots of type $T(p,\,r)$, $T(q,\,s)$, and $T(p+q,\,r+s)$, respectively. 
\end{proof}

\Cref{prop:all-edges-torus-knots,prop:torus-knot-edge} together show that layered $\sthree$ triangulations realise torus knots of prescribed types.
Now we establish the converse, that is, every torus knot arises in a canonical way from such a layered construction.

\begin{proposition}
\label{prop:decomposition}
Let $\pp,\qq \ge 2$ with $\gcd(\pp,\qq)=1$, and consider the torus knot $T(\pp,\qq)$. There exists a unique basis $\{(p,r),(q,s)\}$ of $H_1(T^2;\ZZ)\cong\ZZ^2$ with $p,q,r,s>0$ and
\[
(p,r)+(q,s)=(\pp,\qq),
\]
up to swapping $(p,q,r,s)\leftrightarrow(q,p,s,r)$. In particular, $\pp=p+q$, $\qq=r+s$, and $|p\,s-q\,r|=1$.
\end{proposition}

\begin{proof}
We need integers $p,q,r,s > 0$ such that $p+q=\pp$, $r+s=\qq$, and $|p\,s-q\,r|=1$. Setting $q = \pp - p$ and $s = \qq - r$, the determinant condition
\[
    |p\,s - q\,r|
    \;=\; |p(\qq - r) - (\pp - p)\,r|
    \;=\; |p\,\qq - \pp\,r|
\]
reduces to
\[
    |p\,\qq - \pp\,r| = 1.
\tag{\hellblau{$\ast$}}\label{eq:bezout-pq}
\]

\begin{itemize}
    \item \emph{Existence.} Since $\gcd(\pp,\qq)=1$, B\'ezout's identity furnishes integers $p, r$ satisfying~\Cref{eq:bezout-pq}.  
    
    Reducing $p$ modulo $\pp$ to a representative in $\{1, \dots, \pp - 1\}$ forces $r \in \{1, \dots, \qq - 1\}$, and hence $q = \pp - p$, $s = \qq - r$ are positive.
    \item \emph{Uniqueness.} Reducing~\Cref{eq:bezout-pq} modulo $\pp$ gives
    \[
        p\,\qq \equiv \pm 1 \pmod{\pp}.
    \] 
    The positivity constraint $1 \le p \le \pp - 1$ determines $p$ uniquely as an integer, and the two residues $p$ and $\pp - p$ correspond exactly to the swap $(p,q,r,s) \leftrightarrow (q,p,s,r)$.
\end{itemize}
\end{proof}

\begin{corollary}
\label{cor:canonical-triangulation}
Every non-trivial torus knot $T(\pp,\qq)$ is realised by an edge in a canonical layered one-vertex triangulation $\tri$ of $\sthree$ of the form $\lst(p,q) \, \cup \, \lst(r,s)$, where $\pp=p+q$ and $\qq=r+s$.
\end{corollary}

The size of $\tri$ in the construction of~\Cref{cor:canonical-triangulation} equals the total number of layering steps to build $\lst(p,q)$ and $\lst(r,s)$, which is equal to the number of steps in the subtraction-form Euclidean algorithm performed on both $(p,q)$ and $(r,s)$. This yields an upper bound on the triangulation complexity of $T(\pp,\qq)$ that is at most linear in $\max\{\pp,\qq\}$, and as low as $O(\log (\max\{\pp,\qq\}))$ for certain inputs.

In particular, we obtain at least $3{,}049$ torus knots with a triangulation complexity of up to $19$, and we conjecture this number to be exact (see Conjecture~\ref{conj:sharp}). In contrast, there are only $14$ torus knots in the census of knots up to crossing number $19$ among more than $350$ million hyperbolic~\cite{Burton2020}.

\begin{example}
\label{ex:canonical-triangulations}
We present three families of examples of the canonical triangulation of \Cref{cor:canonical-triangulation}.
    
\begin{itemize}
    \item  \emph{Linear family: $T(2,2k+1)$ with $(p,q,r,s)=(1,1,k,k+1)$.}
    The decomposition of \Cref{prop:decomposition} produces the layered triangulations:

    \begin{center}
        \begin{minipage}{\linewidth}
            \captionsetup{type=table,width=\linewidth}
            \centering
            \begingroup
            \small
            \begin{tabular}{@{}llrrr@{}}
                \toprule
                \textbf{Torus knot}
                    & \textbf{Layered decomposition}
                    & \textbf{Tetrahedra}
                    & $\boldsymbol{\max\{\pp,\qq\}}$
                    & $\boldsymbol{c(K)}$ \\
                \midrule
                $T(2,3)$
                    & $\lst(1,1)\cup\lst(1,2)$
                    & $1$
                    & $3$
                    & $3$ \\
                $T(2,5)$
                    & $\lst(1,1)\cup\lst(2,3)$
                    & $2$
                    & $5$
                    & $5$ \\
                $T(2,7)$
                    & $\lst(1,1)\cup\lst(3,4)$
                    & $3$
                    & $7$
                    & $7$ \\
                \addlinespace[0.35em]
                \midrule
                $T(2,2k+1)$
                    & $\lst(1,1)\cup\lst(k,k+1)$
                    & $k$
                    & $2k+1$
                    & $2k+1$ \\
                \bottomrule
            \end{tabular}

            \endgroup
            \caption{Layered triangulations for the family $T(2,2k+1)$.}
            \label{tab:canonical-linear}
        \end{minipage}
    \end{center}

    The subtraction-form Euclidean algorithm on $(k,k+1)$ needs exactly $k$ steps to descend to $(1,1)$, so $|\lst(k,k+1)|=k$, and the canonical triangulation of $T(2,2k+1)$ contains $k$ tetrahedra.  This is linear in $\max\{\pp,\qq\}=2k+1$ and saturates the worst-case linear growth rate of the construction.

    \item \emph{Logarithmic family: $T(F_n,F_{n+1})$ for even $n \ge 4$, where $F_n$ is the $n$-th Fibonacci number.} Cassini's identity $F_{n-1} F_{n+1} - F_n^2 = 1$~\cite[Eq.~(6.103)]{GKP1994} furnishes the unique decomposition $(p, q, r, s)=(F_{n-1}, F_{n-2}, F_n, F_{n-1})$ of \Cref{prop:decomposition}, yielding the family:

    \begin{center}
        \begin{minipage}{\linewidth}
            \captionsetup{type=table,width=\linewidth}
            \centering
            \begingroup
            \small
            \begin{tabular}{@{}llrrr@{}}
                \toprule
                \textbf{Torus knot}
                    & \textbf{Layered decomposition}
                    & \textbf{Tetrahedra}
                    & $\boldsymbol{\max\{\pp,\qq\}}$
                    & $\boldsymbol{c(K)}$ \\
                \midrule
                $T(3,5)$
                    & $\lst(1,2)\cup\lst(2,3)$
                    & $3$
                    & $5$
                    & $10$ \\
                $T(8,13)$
                    & $\lst(3,5)\cup\lst(5,8)$
                    & $7$
                    & $13$
                    & $91$ \\
                $T(21,34)$
                    & $\lst(8,13)\cup\lst(13,21)$
                    & $11$
                    & $34$
                    & $680$ \\
                $T(55,89)$
                    & $\lst(21,34)\cup\lst(34,55)$
                    & $15$
                    & $89$
                    & $4{,}806$ \\
                $T(144,233)$
                    & $\lst(55,89)\cup\lst(89,144)$
                    & $19$
                    & $233$
                    & $33{,}319$ \\
                \bottomrule
            \end{tabular}

            \endgroup
            \caption{Layered triangulations for torus knots arising from consecutive Fibonacci numbers.}
            \label{tab:canonical-fibonacci}
        \end{minipage}
    \end{center}

    Across consecutive even values of $n$, the tetrahedron count grows by exactly $4$, while $c(K)$ grows exponentially, with consecutive ratios approaching $\phi^{4} \approx 6.854$, where $\phi = (1+\sqrt{5})/2$ is the golden ratio. The Euclidean algorithm therefore runs at its theoretical worst rate~\cite{Lame1844}, but on knots whose crossing number $c(K) = F_{n+1} (F_n - 1)$ is exponential in $n$~\cite[Proposition~7.5]{Murasugi1991}, so the canonical triangulations have size $O(\log\max\{\pp,\qq\}) = O(\log c(K))$.

    \item \emph{Trivial case $\pp=1$.}  The hypothesis $\pp,\qq \ge 2$ in \Cref{prop:decomposition} does not cover this case, and indeed the condition $p,q>0$, $2 \leq p+q=\pp$, causes the unknot $T(1,\qq)$ to be not realised as the central edge $T(p+q,r+s)$ of any canonical triangulation.  It is, however, present among the side edges produced by \Cref{prop:torus-knot-edge}, namely $T(p,r)$ and $T(q,s)$, each of which represents the unknot precisely when one of its two parameters equals $1$.  
      
    In the smallest instance $\lst(1,1)\cup\lst(1,2)$, both side edges are represented by the same \emph{unknotted} edge, so the two edges of this one-tetrahedron triangulation of $\sthree$ realise the unknot and the trefoil, respectively.
    \end{itemize}

    The two extreme families and the resulting difference between crossing-number and canonical-size cutoffs are compared in \Cref{fig:complexity-scatter}.

    \begin{figure}[t!]
        \centering
        \begin{subfigure}[b]{0.49\linewidth}
            \centering
            \includegraphics[
                width=0.98\linewidth,
                height=0.50\linewidth,
                keepaspectratio,
                alt={Log-axis canonical size versus crossing number for the linear and Fibonacci families, shown with circles and triangles. Details in caption.}
            ]{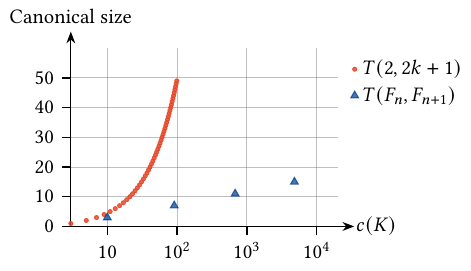}
            \caption{Canonical triangulation size by crossing number.}
            \label{fig:complexity-crossing}
        \end{subfigure}
        \hfill
        \begin{subfigure}[b]{0.49\linewidth}
            \centering
            \includegraphics[
                width=0.98\linewidth,
                height=0.50\linewidth,
                keepaspectratio,
                alt={Log cutoff counts through nineteen; squares and triangles compare canonical size and crossings. Shading marks 83 certified types. Details in caption.}
            ]{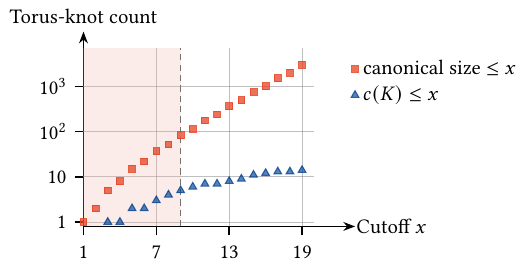}
            \caption{Total number of torus knots by cutoff. }
            \label{fig:complexity-counts}
        \end{subfigure}
        \caption{Canonical triangulation size and cutoff counts. (\ref{sub@fig:complexity-crossing}) The $49$ knots $T(2,2k+1)$ and 4 even-index Fibonacci knots with $2 \le \pp < \qq \le 100$, plotted against crossing number on a logarithmic axis. (\ref{sub@fig:complexity-counts}) A canonical-size cutoff counts far more torus knots than the same \emph{exact} crossing-number cutoff (at $x=19$ the counts are $3{,}049$ and $14$, respectively). Results in the shaded region are verified for size.}
        \label{fig:complexity-scatter}
    \end{figure}
\end{example}

\begin{corollary}
\label{cor:log-complexity}
There exist infinitely many knots with triangulation complexity at most a logarithmic function in their crossing number.
\end{corollary}

\begin{proof}
Such an infinite family is given by $T(F_n, F_{n+1})$ for even $n \ge 4$, where $F_n$ denotes the $n$-th Fibonacci number. The claimed properties follow from the discussion above. 
\end{proof}

\section{Open questions and further work} 

\subsection{Sharpness of the upper bound for torus knots}
\label{sec:open-sharpness}

The construction of \Cref{sec:layered-triangulations} produces a triangulation of $\sthree$ of size
\[
\ell(p,q) + \ell(r,s) \;=\; \lvert\lst(p,q)\rvert + \lvert\lst(r,s)\rvert,
\]
where $(p,q,r,s)$ is the unique decomposition of~\Cref{prop:decomposition}.

\begin{conjecture}
\label{conj:sharp}
Let $\pp,\qq \ge 2$ with $\gcd(\pp,\qq)=1$. The triangulation complexity of $T(\pp,\qq)$ is realised by the canonical construction of \Cref{cor:canonical-triangulation}; that is, it equals $\ell(p,q)+\ell(r,s)$.
\end{conjecture}

A proof of \Cref{conj:sharp} would, via \Cref{cor:log-complexity}, give the first infinite family of knots for which the triangulation complexity is determined exactly and grows logarithmically in the crossing number. The natural lower-bound would follow from the minimality of layered solid tori conjectured by Jaco and Rubinstein~\cite[Conjecture, Section 4.2]{JacoRubinstein2006}, that if every one-vertex triangulation of $\sthree$ containing $T(\pp,\qq)$ as an edge splits along the corresponding boundary torus into two minimally layered solid tori, the upper bound of \Cref{sec:layered-triangulations} matches the lower bound.

\begin{question}
\label{q:splitting}
Does every one-vertex triangulation of $\sthree$ containing a non-trivial torus knot $T(\pp,\qq)$ as an edge admit a Heegaard torus as a subcomplex that decomposes it into two layered solid tori with the edge in the common boundary torus?
\end{question}

Note that, in the above question, a layered solid torus can be a degenerate $0$-tetrahedron M\"obius band. Otherwise, the $1$-tetrahedron triangulation of $\sthree$ containing a trefoil as an edge would already be a counterexample. 

A positive answer to \Cref{q:splitting}, combined with the
Jaco--Rubinstein minimality conjecture, implies \Cref{conj:sharp}.

\subsection{Cable knots}
\label{sec:open-cables}

Given a knot $K: \sone \hookrightarrow \sthree$, we can use a variant of our construction to produce cable knots of $K$. For this, start with a triangulation of the knot exterior of $K$ with boundary a one-vertex torus. By construction, at most one of the boundary edges runs along the boundary of the meridian disk. The other edges are cables of $K$. Applying layerings on the boundary we can introduce edges of every cable type, using the (inverse of the) subtraction version of the Euclidean algorithm. Once the prescribed cable knot is introduced as an edge, we can layer on the resulting boundary to eventually fold to realise the Dehn filling of the exterior to $\sthree$, see also \cite{JacoRubinsteinSpreerTillmann2025}.

\subsection{Satellites}
\label{sec:open-satellites}

In this section we briefly explain how satellite knots, as well, conveniently fit into the triangulation complexity setting.
For this, let $L = P \cup A \subset \sthree$ be a $2$-component link where $P$ describes a pattern and $A$ the boundary of the meridian of a companion knot $C$, called the axis component. We first form the $2$-cusped link exterior
\[
X_{P} = E(L) = \sthree \setminus \operatorname{int} \bigl(N(P) \cup N(A)\bigr).
\]
Since $A$ is unknotted, $\sthree \setminus \operatorname{int} N(A)$ is a solid torus, so $X_{P}$ may be regarded as the exterior of the pattern knot $P$ inside the solid torus following the companion knot $C$. Its two boundary tori are
\[
T_{P}=\partial N(P), \qquad T_{A}=\partial N(A).
\]
We choose a triangulation of a solid torus, where $P$ is realised as an edge (and hence $T_P$ only exist implicitly), and $T_{A}$ is represented as a one-vertex torus with edges along both the meridian $\mu_A$ and the homological longitude $\lambda_A$. We also take the companion exterior
\[
X_{C}=\sthree\setminus \operatorname{int}N(C),
\]
as a triangulation with boundary a one-vertex torus, also with two edges running along meridian $\mu_C$ and homological longitude $\lambda_C$. To build the satellite knot with pattern knot $P$ and companion knot $C$, we glue the two triangulations along their boundaries such that the edge running along the meridian in one boundary component is glued to the edge running along the longitude in the other boundary component, and vice versa:
\[
\phi:T_{A}\longrightarrow \partial X_{C},
\qquad 
\phi(\mu_{A})=\lambda_{C}, 
\qquad
\phi(\lambda_{A})=\mu_{C}.
\]
Note that we can always achieve that edges run along meridian and longitude by first performing an appropriate sequence of layerings on the boundary components $T_{A}$ and $X_{C}$.

The resulting glued manifold
\[
X_{P(C)} = X_{C} \cup_{\phi}X_{P}
\]
is the exterior of the satellite knot with pattern knot $P$ and companion knot $C$ represented by the ideal edge in the triangulation of $X_{P}$. 

We can iterate this construction, by identifying another axis boundary component in the triangulation of $X_{P(C)}$. This construction can be made more efficient, if, in a pre-processing step, we pinch together vertices on distinct (edge ideal or real) boundary components.

\subsection{Knots in lens spaces}
\label{sec:open-lens}

\Cref{prop:torus-knot-edge} requires a layered one-vertex triangulation of a closed $3$-manifold to admit a genus-$1$ Heegaard splitting.  Replacing a layered triangulation of $\sthree$ by a layered triangulation of an arbitrary lens space $L(\lensp,\lensq)$ enables an analogous construction for edges representing torus knots in lens spaces.

\begin{question}
\label{q:lens}
For which pairs $\bigl( L(\lensp, \lensq), \, K \bigr)$ with $K \subset L(\lensp,\lensq)$ a knot is $K$ realised as an edge in a layered one-vertex triangulation of $L(\lensp,\lensq)$, and what is the resulting triangulation complexity of the pair?
\end{question}

The case $\lensp=0$, corresponding to $\stwo \times \sone$ via the standard genus-$1$ Heegaard splitting, is already nontrivial and would describe knots in $\stwo \times \sone$ of bounded triangulation complexity in terms of Farey combinatorics, in the spirit of~\cite{JacoRubinstein2006}.

\bibliographystyle{amsalpha}
\bibliography{references}

@misc{Regina,
    author       = {Burton, Benjamin A. and Budney, Ryan and Pettersson, William and others},
    title        = {Regina: Software for low-dimensional topology},
    howpublished = {\url{https://regina-normal.github.io/}},
    year         = {1999--2025},
}

@article{Brocot1861,
    author  = {Brocot, Achille},
    title   = {Calcul des rouages par approximation, nouvelle m\'ethode},
    journal = {Revue Chronom\'etrique},
    volume  = {3},
    year    = {1861},
    pages   = {186--194}
}

@inproceedings{Burton2011,
    author    = {Burton, Benjamin A.},
    title     = {Detecting genus in vertex links for the fast enumeration of 3-manifold triangulations},
    year      = {2011},
    isbn      = {9781450306751},
    publisher = {Association for Computing Machinery},
    address   = {New York, NY, USA},
    url       = {https://doi.org/10.1145/1993886.1993901},
    doi       = {10.1145/1993886.1993901},
    booktitle = {Proceedings of the 36th International Symposium on Symbolic and Algebraic Computation},
    pages     = {59–66},
    numpages  = {8},
    location  = {San Jose, California, USA},
    series    = {ISSAC '11}
}

@InProceedings{Burton2020,
    author    =	{Burton, Benjamin A.},
    title     =	{{The Next 350 Million Knots}},
    booktitle =	{36th International Symposium on Computational Geometry (SoCG 2020)},
    pages     =	{25:1--25:17},
    series    =	{Leibniz International Proceedings in Informatics (LIPIcs)},
    ISBN      =	{978-3-95977-143-6},
    ISSN      =	{1868-8969},
    year      =	{2020},
    volume    =	{164},
    editor    =	{Cabello, Sergio and Chen, Danny Z.},
    publisher =	{Schloss Dagstuhl -- Leibniz-Zentrum f{\"u}r Informatik},
    address   =	{Dagstuhl, Germany},
    URL       =		{https://drops.dagstuhl.de/entities/document/10.4230/LIPIcs.SoCG.2020.25},
    URN       = {urn:nbn:de:0030-drops-121831},
    doi       = {10.4230/LIPIcs.SoCG.2020.25}
}

@misc{SnapPy,
    author       = {Culler, Marc and Dunfield, Nathan M. and Goerner, Matthias and Weeks, Jeffrey R.},
    title        = {SnapPy, a computer program for studying the geometry and topology of $3$-manifolds},
    howpublished = {Available at \url{https://snappy.computop.org/}},
    note         = {Version 3.2}
}

@book{Fowler1987,
    author    = {Fowler, David H.},
    title     = {The Mathematics of {P}lato's {A}cademy: a new reconstruction},
    publisher = {Clarendon Press},
    address   = {Oxford},
    year      = {1987},
    isbn      = {0-19-853912-6}
}

@book{GKP1994,
    author    = {Graham, Ronald L. and Knuth, Donald E. and Patashnik, Oren},
    title     = {Concrete Mathematics: A Foundation for Computer Science},
    edition   = {2},
    publisher = {Addison-Wesley},
    address   = {Reading, MA},
    year      = {1994},
    isbn      = {0-201-55802-5}
}

@book{HatcherTopologyOfNumbers,
    author    = {Hatcher, Allen},
    title     = {Topology of Numbers},
    publisher = {American Mathematical Society},
    address   = {Providence, RI},
    year      = {2022},
    isbn      = {978-1-4704-5611-5},
    url       = {https://pi.math.cornell.edu/~hatcher/TN/TNbook.pdf}
}

@book{Hempel1976,
    author    = {Hempel, John},
    title     = {3-{M}anifolds},
    publisher = {Princeton University Press},
    address   = {Princeton, NJ},
    series    = {Annals of Mathematics Studies},
    number    = {86},
    year      = {1976}
}

@article{HassLagariasPippenger1999,
    author     = {Hass, Joel and Lagarias, Jeffrey C. and Pippenger, Nicholas},
    title      = {The computational complexity of knot and link problems},
    year       = {1999},
    issue_date = {March 1999},
    publisher  = {Association for Computing Machinery},
    address    = {New York, NY, USA},
    volume     = {46},
    number     = {2},
    issn       = {0004-5411},
    url        = {https://doi.org/10.1145/301970.301971},
    doi        = {10.1145/301970.301971},
    journal    = {Journal of the ACM},
    month      = mar,
    pages      = {185–211},
    numpages   = {27},
}

@article{HeMorganThompson2025,
    author        = {He, Alexander and Morgan, James and Thompson, Em K.},
    title         = {An algorithm to construct one-vertex triangulations of {H}eegaard splittings},
    volume        = {16}, 
    url           = {https://jocg.org/index.php/jocg/article/view/4779}, 
    DOI           = {10.20382/jocg.v16i1a18},
    number        = {1},
    journal       = {Journal of Computational Geometry},
    year          = {2025},
    month         = {Nov.},
    pages         = {635--693},
    eprint        = {2312.17556},
    archiveprefix = {arXiv},
    primaryclass  = {math.GT}
}

@inproceedings{HuszarSpreer2019,
    author        =	{Husz\'{a}r, Krist\'{o}f and Spreer, Jonathan},
    title         =	{{3-Manifold Triangulations with Small Treewidth}},
    booktitle     = {35th International Symposium on Computational Geometry ({SoCG} 2019)},
    pages         =	{44:1--44:20},
    series        = {Leibniz International Proceedings in Informatics (LIPIcs)},
    ISBN          =	{978-3-95977-104-7},
    ISSN          =	{1868-8969},
    year          = {2019},
    volume        = {129},
    editor        =	{Barequet, Gill and Wang, Yusu},
    publisher     =	{Schloss Dagstuhl -- Leibniz-Zentrum f{\"u}r Informatik},
    address       =	{Dagstuhl, Germany},
    URL           = {https://drops.dagstuhl.de/entities/document/10.4230/LIPIcs.SoCG.2019.44},
    URN           = {urn:nbn:de:0030-drops-104487},
    doi           =	{10.4230/LIPIcs.SoCG.2019.44},
    eprint        = {1812.05528},
    archiveprefix = {arXiv},
    primaryclass  = {math.GT}
}

@inproceedings{HeSedgwickSpreer2025,
    author        = {He, Alexander and Sedgwick, Eric and Spreer, Jonathan},
    title         = {A Practical Algorithm for Knot Factorisation},
    booktitle     = {41st International Symposium on Computational Geometry (SoCG 2025)},
    pages         = {55:1--55:15},
    series        = {Leibniz International Proceedings in Informatics (LIPIcs)},
    ISBN          =	{978-3-95977-370-6},
    ISSN          =	{1868-8969},
    year          = {2025},
    volume        = {332},
    editor        =	{Aichholzer, Oswin and Wang, Haitao},
    publisher     = {Schloss Dagstuhl -- Leibniz-Zentrum f\"ur Informatik},
    address       =	{Dagstuhl, Germany},
    URL           = {https://drops.dagstuhl.de/entities/document/10.4230/LIPIcs.SoCG.2025.55},
    URN           = {urn:nbn:de:0030-drops-232075},
    doi           = {10.4230/LIPIcs.SoCG.2025.55},
    eprint        = {2504.03942},
    archiveprefix = {arXiv},
    primaryclass  = {math.GT}
}

@misc{IbarraMathewsPurcellSpreer2024,
    author        = {Ibarra, Dionne and Mathews, Daniel V. and Purcell, Jessica S. and Spreer, Jonathan},
    title         = {Triangulations of the {$3$}-sphere with knotted edge},
    year          = {2024},
    eprint        = {2411.18938},
    archiveprefix = {arXiv},
    primaryclass  = {math.GT},
    url           = {https://arxiv.org/abs/2411.18938},
    note          = {Preprint, \href{https://arxiv.org/abs/2411.18938}{\texttt{arXiv:2411.18938}}}
}

@misc{JacoRubinstein2006,
    author        = {Jaco, William and Rubinstein, J. Hyam},
    title         = {Layered-triangulations of 3-manifolds},
    year          = {2006},
    eprint        = {math/0603601},
    archiveprefix = {arXiv},
    primaryclass  = {math.GT},
    url           = {https://arxiv.org/abs/math/0603601},
    note          = {Preprint, \href{https://arxiv.org/abs/math/0603601}{\texttt{arXiv:math/0603601}}}
}

@article{JacoRubinsteinSpreerTillmann2020,
    author        = {Jaco, William and Rubinstein, J. Hyam and Spreer, Jonathan and Tillmann, Stephan},
    title         = {{$\mathbb{Z}_2$}-{T}hurston norm and complexity of 3-manifolds, {II}},
    journal       = {Algebraic \& Geometric Topology},
    volume        = {20},
    number        = {1},
    pages         = {503--529},
    year          = {2020},
    doi           = {10.2140/agt.2020.20.503},
    eprint        = {1711.10737},
    archiveprefix = {arXiv},
    primaryclass  = {math.GT}
}

@article{JacoRubinsteinSpreerTillmann2025,
    author  = {Jaco, William and Rubinstein, J. Hyam and Spreer, Jonathan and Tillmann, Stephan},
    title   = {Complexity of 3-manifolds obtained by {D}ehn filling},
    journal = {Algebraic \& Geometric Topology},
    volume  = {25},
    number  = {1},
    year    = {2025},
    pages   = {301--327},
    doi     = {10.2140/agt.2025.25.301}
}

@article{JacoRubinsteinTillmann2009,
    author        = {Jaco, William and Rubinstein, J. Hyam and Tillmann, Stephan},
    title         = {Minimal triangulations for an infinite family of lens spaces},
    journal       = {Journal of Topology},
    volume        = {2},
    number        = {1},
    pages         = {157--180},
    doi           = {https://doi.org/10.1112/jtopol/jtp004},
    url           = {https://londmathsoc.onlinelibrary.wiley.com/doi/abs/10.1112/jtopol/jtp004},
    year          = {2009},
    eprint        = {0805.2425},
    archiveprefix = {arXiv},
    primaryclass  = {math.GT}
}

@article{Kauffman1987,
    author  = {Kauffman, Louis H.},
    title   = {State models and the jones polynomial},
    journal = {Topology},
    volume  = {26},
    number  = {3},
    pages   = {395--407},
    year    = {1987},
    issn    = {0040-9383},
    doi     = {https://doi.org/10.1016/0040-9383(87)90009-7},
    url     = {https://www.sciencedirect.com/science/article/pii/0040938387900097},
}

@article{KronheimerMrowkaOzsvathSzabo2007,
    author        = {Kronheimer, Peter B. and Mrowka, Tomasz S. and Ozsv\'ath, Peter and Szab\'o, Zolt\'an},
    title         = {Monopoles and lens space surgeries},
    journal       = {Annals of Mathematics},
    series        = {2},
    volume        = {165},
    number        = {2},
    year          = {2007},
    pages         = {457--546},
    doi           = {https://doi.org/http://doi.org/10.4007/annals.2007.165.457},
    eprint        = {math/0310164},
    archiveprefix = {arXiv},
    primaryclass  = {math.GT}
}

@incollection{Lackenby2020,
    author        = {Lackenby, Marc},
    title         = {Algorithms in 3-manifold theory},
    booktitle     = {Surveys in 3-manifold topology and geometry},
    editor        = {Agol, Ian and Gabai, David},
    series        = {Surveys in Differential Geometry},
    volume        = {25},
    number        = {1},
    publisher     = {International Press},
    address       = {Boston, MA},
    year          = {2020},
    pages         = {163--213},
    doi           = {10.4310/SDG.2020.v25.n1.a5},
    eprint        = {2002.02179},
    archiveprefix = {arXiv},
    primaryclass  = {math.GT}
}

@article{Lame1844,
    author  = {Lam\'e, Gabriel},
    title   = {Note sur la limite du nombre des divisions dans la recherche du plus grand commun diviseur entre deux nombres entiers},
    journal = {Comptes rendus de l'Acad\'emie des sciences},
    volume  = {19},
    year    = {1844},
    pages   = {867--870}
}

@article{Lickorish1962,
    author  = {Lickorish, W. B. R.},
    title   = {A representation of orientable combinatorial 3-manifolds},
    journal = {Annals of Mathematics},
    series  = {2},
    volume  = {76},
    number  = {3},
    year    = {1962},
    pages   = {531--540},
    doi     = {10.2307/1970373}
}

@misc{Lin2026Github,
    author       = {Lin, Lezhi},
    title        = {{TorusKnots}},
    howpublished = {Available at \url{https://github.com/HimalayanRainstorm/TorusKnots}},
    year         = {2026},
}

@article{LackenbyPurcell2024,
    author        = {Lackenby, Marc and Purcell, Jessica S.},
    title         = {The triangulation complexity of fibred 3-manifolds},
    journal       = {Geometry \& Topology},
    volume        = {28},
    number        = {4},
    pages         = {1727--1828},
    year          = {2024},
    doi           = {10.2140/gt.2024.28.1727},
    eprint        = {1910.10914},
    archiveprefix = {arXiv},
    primaryclass  = {math.GT}
}

@book{Matveev2007,
    author    = {Matveev, Sergei},
    title     = {Algorithmic Topology and Classification of $3$-Manifolds},
    publisher = {Springer},
    address   = {Berlin},
    year      = {2007},
    edition   = {2},
    series    = {Algorithms and Computation in Mathematics},
    volume    = {9}
}

@article{Moise1952,
    author  = {Moise, Edwin E.},
    title   = {Affine Structures in $3$-Manifolds. {V}. {T}he Triangulation Theorem and {H}auptvermutung},
    journal = {Annals of Mathematics},
    series  = {2},
    volume  = {56},
    year    = {1952},
    pages   = {96--114},
    doi     = {10.2307/1969769}
}

@article{Moriah1988,
    author   = {Moriah, Yoav},
    title    = {Heegaard splittings of {S}eifert fibered spaces},
    journal  = {Inventiones Mathematicae},
    volume   = {91},
    number   = {3},
    pages    = {465--481},
    year     = {1988},
    doi      = {10.1007/BF01388781},
    isbn     = {1432-1297},
    mrnumber = {928492}
}

@article{Moser1971,
    author  = {Moser, Louise},
    title   = {Elementary surgery along a torus knot},
    journal = {Pacific Journal of Mathematics},
    volume  = {38},
    number  = {3},
    year    = {1971},
    pages   = {737--745},
    doi     = {10.2140/pjm.1971.38.737}
}

@article{Murasugi1991,
    author  = {Murasugi, Kunio},
    title   = {On the braid index of alternating links},
    journal = {Transactions of the American Mathematical Society},
    volume  = {326},
    number  = {1},
    year    = {1991},
    pages   = {237--260},
    doi     = {10.1090/S0002-9947-1991-1000333-3}
}

@article{OzsvathSzabo2005,
    author  = {Ozsv\'ath, Peter and Szab\'o, Zolt\'an},
    title   = {On knot {F}loer homology and lens space surgeries},
    journal = {Topology},
    volume  = {44},
    number  = {6},
    year    = {2005},
    pages   = {1281--1300},
    doi     = {10.1016/j.top.2005.05.001}
}

@article{Reidemeister1935,
    author  = {Reidemeister, Kurt},
    title   = {Homotopieringe und {L}insenr\"aume},
    journal = {Abhandlungen aus dem Mathematischen Seminar der Universit\"at Hamburg},
    volume  = {11},
    number  = {1},
    pages   = {102--109},
    year    = {1935},
    doi     = {10.1007/BF02940717}
}

@book{Rolfsen2003,
    author    = {Rolfsen, Dale},
    title     = {Knots and Links},
    publisher = {AMS Chelsea Publishing},
    address   = {Providence, RI},
    year      = {2003},
    volume    = {346}
}

@book{Saveliev2011,
    author    = {Saveliev, Nikolai},
    title     = {Lectures on the Topology of 3-Manifolds: An Introduction to the Casson Invariant},
    edition   = {2},
    series    = {De Gruyter Textbook},
    publisher = {De Gruyter},
    address   = {Berlin},
    year      = {2011},
    isbn      = {978-3-11-025035-0},
    doi       = {10.1515/9783110250367}
}

@article{Schultens1993,
    author  = {Schultens, Jennifer},
    title   = {The classification of {H}eegaard splittings for (compact orientable surface)$\,\times\,S^1$},
    journal = {Proceedings of the London Mathematical Society},
    volume  = {s3-67},
    number  = {2},
    year    = {1993},
    pages   = {425--448},
    doi     = {https://doi.org/10.1112/plms/s3-67.2.425},
    url     = {https://londmathsoc.onlinelibrary.wiley.com/doi/abs/10.1112/plms/s3-67.2.425},
    eprint  = {https://londmathsoc.onlinelibrary.wiley.com/doi/pdf/10.1112/plms/s3-67.2.425},
}

@article{Stern1858,
    author  = {Stern, Moritz A.},
    title   = {{\"U}ber eine zahlentheoretische {F}unktion},
    journal = {Journal f\"ur die reine und angewandte Mathematik},
    volume  = {55},
    year    = {1858},
    pages   = {193--220},
    doi     = {10.1515/crll.1858.55.193}
}

@article{Teragaito2004,
    author        = {Teragaito, Masakazu},
    title         = {Crosscap numbers of torus knots},
    journal       = {Topology and its Applications},
    volume        = {138},
    number        = {1},
    pages         = {219--238},
    year          = {2004},
    issn          = {0166-8641},
    doi           = {https://doi.org/10.1016/j.topol.2003.08.004},
    url           = {https://www.sciencedirect.com/science/article/pii/S0166864103002840},
}

@article{Thurston1982,
    author  = {Thurston, William P.},
    title   = {Three dimensional manifolds, {K}leinian groups and hyperbolic geometry},
    journal = {Bulletin of the American Mathematical Society},
    series  = {New Series},
    volume  = {6},
    number  = {3},
    year    = {1982},
    pages   = {357--381},
    doi     = {10.1090/S0273-0979-1982-15003-0}
}

@book{Thurston1997,
    author    = {Thurston, William P.},
    title     = {Three-Dimensional Geometry and Topology. {V}olume 1},
    publisher = {Princeton University Press},
    address   = {Princeton, NJ},
    year      = {1997},
    isbn      = {9780691083049},
    note      = {Edited by Silvio Levy} 
}

@article{Tietze1908,
    author  = {Tietze, Heinrich},
    title   = {{\"U}ber die topologischen {I}nvarianten mehrdimensionaler {M}annigfaltigkeiten},
    journal = {Monatshefte f\"ur Mathematik und Physik},
    volume  = {19},
    number  = {1},
    pages   = {1--118},
    year    = {1908},
    doi     = {10.1007/BF01736688}
}

@article{Wallace1961,
    author  = {Wallace, Andrew H.},
    title   = {Modifications and cobounding manifolds},
    journal = {Canadian Journal of Mathematics},
    volume  = {12},
    year    = {1960},
    pages   = {503--528},
    doi     = {10.4153/CJM-1960-045-7}
}

@article{Whitehead1941,
    author    = {Whitehead, J. H. C.},
    title     = {{On Incidence Matrices, Nuclei and Homotopy Types}},
    journal   = {Annals of Mathematics},
    publisher = {[Annals of Mathematics, Trustees of Princeton University on Behalf of the Annals of Mathematics, Mathematics Department, Princeton University]},
    volume    = {42},
    number    = {5},
    pages     = {1197--1239},
    year      = {1941},
    ISSN      = {0003486X, 19398980},
    URL       = {http://www.jstor.org/stable/1970465}, 
}

\appendix

\section*{Source Code}

An implementation of our procedure is available from \cite{Lin2026Github}. The authors would also like to explicitly point to the very useful software packages {\em Regina} \cite{Regina} and {\em SnapPy} \cite{SnapPy}.

\end{document}